\documentclass[11pt]{amsart}
\usepackage{amsmath}
\usepackage[active]{srcltx}
\usepackage{t1enc}
\usepackage[latin2]{inputenc}
\usepackage{verbatim}
\usepackage{amsmath,amsfonts,amssymb,amsthm}
\usepackage[mathcal]{eucal}
\usepackage{enumerate}
\usepackage[centertags]{amsmath}
\usepackage{graphics}
\usepackage[active]{srcltx}

\newtheorem{theorem}{Theorem}
\newtheorem{lemma}{Lemma}
\newtheorem{cor}{Corollary}

\newtheorem{OldTheorem}{Theorem}

\def\Log{{\rm Log\,}}

\def\tg{{\rm tg \,}}
\def\ctg{{\rm ctg \,}}

\def\ZR{\ensuremath{\mathbb R}}
\def\ZZ{\ensuremath{\mathbb Z}}

\def\ZN{\ensuremath{\mathbb N}}
\def\ZT{\ensuremath{\mathbb T}}

\def\md#1#2\emd{
\ifx0#1\begin{equation*} #2 \end{equation*}\fi
\ifx1#1\begin{equation}#2\end{equation}\fi
\ifx2#1\begin{align*}#2\end{align*}\fi
\ifx3#1\begin{align}#2\end{align}\fi
\ifx4#1\begin{gather*}#2\end{gather*}\fi
\ifx5#1\begin{gather}#2\end{gather}\fi
\ifx6#1\begin{multline*}#2\end{multline*}\fi
\ifx7#1\begin{multline}#2\end{multline}\fi
\ifx8#1\begin{multline*}\begin{split}#2\end{split}\end{multline*}\fi
\ifx9#1\begin{multline}\begin{split}#2\end{split}\end{multline}\fi
}
\newcommand {\e }[1]{(\ref{#1})}
\newcommand {\lem }[1]{Lemma \ref{#1}}
\newcommand {\trm }[1]{Theorem \ref{#1}}
\newcommand {\coro }[1]{Corollary \ref{#1}}

\begin{document}

\author{G. A. Karagulyan and H. Mkoyan}
\title[An exponential estimate for square partial sums]{An exponential estimate for the square partial sums of multiple Fourier series}

\address{G. A. Karagulyan, Faculty of Mathematics and Mechanics, Yerevan
	State University, Alex Manoogian, 1, 0025, Yerevan, Armenia}
\email{g.karagulyan@ysu.am}
\address{H. 	Mkoyan, Faculty of Mathematics and Mechanics, Yerevan
	State University, Alex Manoogian, 1, 0025, Yerevan, Armenia}
\email{hasmikmkoyan@ysu.am}

\maketitle

\begin{abstract}
We prove an exponential integral estimate for the quadratic partial sums
of multiple Fourier series on large sets that implies some new properties of Fourier series. 
\end{abstract}

\medskip

\footnotetext{%
}

\section{Introduction}

Let $\ZT=\ZR/2\pi$ and $\ZT^d$ denote the $d$-dimensional torus.
The multiple trigonometric Fourier series of a function $f\in L^1(\ZT^d)$
and its conjugate are the series
\begin{align}
&\sum_{\textbf{n}= (n_1,\ldots,n_d) \in\ZZ^d }a_{\textbf{n}} e^{i\textbf
	n\cdot \textbf{x}},\label{x1}\\
&\sum_{\textbf{n}= (n_1,\ldots,n_d) \in\ZZ^d\setminus \{{\bf 0}\}}\overline a_{\textbf{n}} e^{-i\textbf
	n\cdot \textbf{x}},\label{x2}
\end{align}%
where
\begin{align*}
&\textbf{n}=(n_1,\ldots,n_d),\quad \textbf{x}=(x_1,\ldots,x_d),\\
&\textbf n\cdot \textbf{x}=n_1x_1+\ldots+n_dx_d,\\
&a_{\textbf{n}}=\frac{1}{(2\pi)^d }\int_{\ZT^d}f\left(
\textbf{x}\right) e^{-i\textbf{n}\cdot \textbf{x}}d\textbf{x}.
\end{align*}%
Denote the rectangular and square partial sums of series \e {x1} by
\begin{align*}
&S_{\textbf{n}}f\left({\bf x}\right) =\sum_{-n_i\le k_i\le n_i}a_{\textbf{k}} e^{i\textbf
			k\cdot \textbf{x}},\quad \textbf{n}\in\ZZ^d,\\
&S_{n}f\left( {\bf x}\right) =\sum_{-n\le k_i\le n}a_{\textbf{k}} e^{i\textbf
		k\cdot \textbf{x}},\quad n\in\ZN,
\end{align*}%
and let $\tilde S_{\textbf{n}}$ and $\tilde S_{n}$ be their conjugates respectively. 

We shall consider the Orlicz classes of functions corresponding to the logarithmic functions 
\begin{equation}
\Log_k(u)=|u|\max\{0,\log^k|u|\},\quad k=1,2,\ldots .
\end{equation} 
That is the Banach space of functions
\begin{equation*}
\Log_k(L)(\ZT^d)=\left\{f\in L^1({\ZT^d}):\, \int_{\ZT^d}\Log_k(f)<\infty\right\}
\end{equation*}
with the Luxemburg norm
\begin{equation*}
\| f\|_{\Log_k(L)}=\inf \left\{ \lambda :\,\lambda
>0,\,\int\limits_{\ZT^d} \Log_k\left( \frac{f}{\lambda }\right) \leq
1\right\} <\infty.
\end{equation*}
It is well known that the rectangular partial sums of $d$-dimensional Fourier series of any function $f\in \Log_{d-1}(L)(\ZT^d)$ converge in measure  (\cite{Zyg1}, \cite{Zhi}), that means we have
\begin{equation}\label{a75}
\lim_{\min({\bf n})\to \infty}\left|\{{\bf x}\in\ZT^d:\,|S_{\textbf{n}}f\left({\bf x}\right)-f({\bf x})|>\varepsilon \}\right|=0 
\end{equation}
for any $\varepsilon>0$, where 
\begin{equation*}
\min({\bf n})=\min_{1\le i\le d}{n_i}.
\end{equation*}
On the other hand it was established by Konyagin \cite {Kon} and Getsadze \cite{Get} that $\Log_{d-1}(L)$ is the widest  Orlicz space, whose functions satisfy \e {a75}. 

The papers \cite {Kar1, Kar2} have considered the following problem: find the exact estimate for the growth of the function $\Phi:\ZR^+\to\ZR^+$ with $\lim_{t\to\infty}\Phi(t)=0$ such that for any function $f\in \Log_{d-1}(L)(\ZT^d)$ and $\varepsilon>0$ one can find a set $E_{f,\varepsilon}\subset \ZT^d$, $|E_{f,\varepsilon}|>(2\pi)^d-\varepsilon$, satisfying the condition
\begin{equation}\label{a36}
\lim_{\min({\bf n})\to \infty}\int_{E_{f,\varepsilon}}\Phi(|S_{\bf n}f({\bf x})-f({\bf x})|)d{\bf x}=0.
\end{equation}
The expected sharp bound of the rate of such function is 
\begin{equation}\label{a37}
\limsup_{t\to\infty}\frac{\log \Phi(t)}{t^{1/d}}<\infty. 
\end{equation}
One can observe that relation \e {a36} implies convergence in measure and moreover, it gives quantitative characterization  of the convergence rate. 

In one dimension this problem was considered in \cite{Kar2}. It was proved in \cite{Kar2} the estimate
\begin{equation}\label{a39}
\int_{\ZT}\exp\left(c_1\frac{\tilde f(x) }{Mf(x)}\right)dx<c_2
\end{equation}
for the conjugate function $\tilde f$, where $Mf(x)$ is the Hardy-Littlewood maximal function. Applied this inequality the paper derives the following exponential estimate for the one dimensional partial sums of Fourier series, which in turn immediately implies \e {a36} in one dimension. 
\begin{OldTheorem}[\cite {Kar2}]	
	For any $f\in L^1(\ZT)$ it holds the inequality
	\begin{equation}\label{a67}
	\int_{\ZT}\exp\left(c_1\frac{|S_nf(x)|+|\tilde S_nf(x)|}{Mf(x)}\right)dx\le c_2,\quad n=1,2,\ldots,
	\end{equation}
where $c_1$ and $c_2$ are absolute constants.
\end{OldTheorem}
The sharpness of exponent in \e {a67} and so in \e {a36} was proved by Oskolkov \cite{Osk}.

The relation \e {a36} in two dimensions with $\Phi$ satisfying \e {a37} was established in \cite{Kar1}. The case $d\ge 3$ of the problem as well as the problem of sharpness  of condition \e {a37} in two dimensions are open. 

The analogous estimates for one-dimensional Walsh and rearranged Haar systems were proved in \cite{Kar3}. The paper \cite{Kar4} considers a similar problem for general orthogonal $L^2$-series. 

In this paper we consider a similar problem for the square partial sums. The main results of the paper is the following
\begin{theorem}\label {T1}
	For any $f\in \Log_{d-1}(L)(\ZT^d)$ there exists a measurable function $F(\textbf{x})>0$  on $\ZT^d$ such that
	\begin{align}
	&|\{\textbf{x}\in\ZT^d:\,F(\textbf{x})>\lambda \}|\lesssim \frac{\|f\|_{\Log_{d-1}(\ZT^d)}}{\lambda},\label{a47}\\
	&\int_{\ZT^d}\exp\left(\frac{|S_nf({\bf x})|+|\tilde S_nf({\bf x})|}{F(\textbf{x})}\right)d{\bf x}\lesssim 1,\quad n=1,2,\ldots.\label{a40}
	\end{align}
\end{theorem}
The relation $a\lesssim b$ in the theorem and in the sequel stands for the inequality $a\leq c\cdot b$, where
$c$ is a constant that can depend only on the dimension $d$.
\begin{cor}\label{C1}
	For any $f\in \Log_{d-1}(L)(\ZT^d)$ and $\varepsilon>0$ there exists a set $E=E_{f,\varepsilon}\subset \ZT^d$ such that
	\begin{align}
	&|E_{f,\varepsilon}|>(2\pi)^d-\varepsilon,\label{a54}\\
	&\int_{E_{f,\varepsilon}}\exp\left(\gamma\varepsilon\frac{ |S_nf({\bf x})|+|\tilde S_nf({\bf x})|}{\|f\|_{\Log_{d-1}(\ZT^d)}}\right)d{\bf x}\lesssim 1,\quad n=1,2,\ldots,\label{a27}
	\end{align}
	where $\gamma>0$ is a constants depended on the dimension $d$.
\end{cor}
\begin{cor}\label{C2}
For any $f\in\Log_{d-1}(L)(\ZT^d)$ and $\varepsilon >0$ there exists a set $E_{f,\varepsilon}\subset \ZT^d$ such that the relations
\begin{align}
&\lim_{n\to \infty}\int_{E_{f,\varepsilon}}(\exp(A|S_nf({\bf x})-f({\bf x})|)-1)d{\bf x}=0,\label{a58}\\
&\lim_{n\to \infty}\int_{E_{f,\varepsilon}}(\exp(A|\tilde S_nf({\bf x})-\tilde f({\bf x})|)-1)d{\bf x}=0.\label{a59}
\end{align}
hold for any $A>0$.
\end{cor}
\begin{cor}\label{C3}
Let $n_k$ be an arbitrary sequence of integers. Then every function $f\in \Log_{d-1}(L)(\ZT^d)$ satisfies the relations
\begin{equation}\label{a66}
S_{n_k}f({\bf x})=o(\log k),\quad \tilde S_{n_k}f({\bf x})=o(\log k) 
\end{equation}
almost everywhere as $k\to\infty$.
\end{cor}
Note that the counterexamples given by Konyagin \cite {Kon} and Getsadze \cite{Get} prove that the space $\Log_{d-1}(L)(\ZT^d)$ is the widest Orlicz class, where such properties hold. 

We prove \trm {T1} reducing it to the one dimensional case. This is a well known argument first used by Sj\"{o}lin in \cite{Sjo} in the proof of multiple version of Carleson's theorem. 

\section{Notations and lemmas}\label{S2}
According to a theorem from \cite{KrRu} (see chap. 2, theorem 9.5) the Luxemburg norm satisfies the relations
\begin{align}  
\| f\|_{\Log_k(L)}\le 1\Rightarrow \int_{\ZT^d}\Log_k(f)\le \| f\|_{\Log_k(L)},\label{a70}\\
\| f\|_{\Log_k(L)}\ge 1\Rightarrow \int_{\ZT^d}\Log_k(f)\ge \| f\|_{\Log_k(L)}.\label{a71}
\end{align}
In fact, such inequalities hold not only for logarithmic, but as well for general Luxemburg norms. Applying \e {a70} and \e {a71}, one can easily check that for any $f\in\Log_k(\ZT^d)$ it holds the inequality
\begin{equation}\label{a72}
\| f\|_{\Log_k(L)}\lesssim 1+\int_{\ZT^d}\Log_k(f).
\end{equation}
Besides, if in addition $\| f\|_{\Log_k(L)}=1$, then we have both upper and lower bounds
\begin{equation}\label{a73}
1+\int_{\ZT^d}\Log_k(f)\lesssim \| f\|_{\Log_k(L)}=1\lesssim 1+\int_{\ZT^d}\Log_k(f).
\end{equation}
The one dimensional conjugate function of $f\in L^1(\ZT)$ is defined by
\begin{equation}\label{a68}
\tilde f(x)=\text{p.v.}\frac{1}{\pi }\int_\ZT \frac{f(x+t)}{2\tg(t/2)}dt=\lim_{\varepsilon\to 0}\frac{1}{\pi }\int_{\varepsilon<|t|<\pi} \frac{f(x+t)}{2\tg(t/2)}dt.
\end{equation}
It is well known that $\tilde f(x)$ is a.e. defined for Lebesgue integrable functions and it holds the inequality 
\md1\label{Zygmund}
\int_\ZT\Log_{k-1}(\tilde f)\lesssim 1+\int_\ZT\Log_{k}(f),\quad k=1,2,\ldots ,
\emd
(see \cite{Zyg2}, chap. 7). We will need this inequality in the following form.
\begin{lemma}\label{L1}
	If $f\in \Log_k(L)(\ZT^d)$, $k=0,1,\ldots $, then the function
	\begin{equation*}
	g(x_1,x_2,\ldots,x_d)=\text{p.v.}\int_{\ZT}\frac{f(x_1+t,x_2+t,x_3,\ldots, x_d)}{\tg(t/2)}dt
	\end{equation*}
is a.e. defined on $\ZT^d$ and satisfies the bound
\begin{equation*}
\int_{\ZT^d}\Log_{k-1}(\tilde g)\lesssim 1+\int_{\ZT^d}\Log_{k}(f).
\end{equation*}
\end{lemma}
The $d$-dimensional conjugate of a function $f\in L^1(\ZT^d)$ is the consecutive application of \e {a68} with respect to each variables of $f$. That is
\begin{equation}\label{a69}
\tilde f(\textbf{x})=\text{p.v.}\frac{1}{\pi}\int_\ZT\ldots \text{p.v.}\frac{1}{\pi}\int_\ZT f(\textbf{x}+\textbf{t})dt_1\ldots dt_d.
\end{equation}
The $d$-dimensional conjugate function $\tilde f$ is a.e. defined for any $f\in \Log_{d-1}(\ZT^d)$. Note that the function $\tilde f(\textbf{x})$ remains the same with respect to any order of integrations in \e {a69}. In the sequel all the integrals will be understood in the sense of principal value and we will omit the notation p.v. before the integrals. Two dimensional case of the following lemma was proved in \cite {GGK}.    
It makes possible to use the modified partial sums 
\begin{align*}
&S_n^*f({\bf x})=\frac{1}{\pi^d}\int_{\ZT^d}\prod_{k=1}^d\frac{\sin nt_k}{2\tg (t_k/2)}f({\bf x}+{\bf t})d{\bf t},\\
&\tilde S_n^*f({\bf x})=\frac{1}{\pi^d}\int_{\ZT^d}\prod_{k=1}^d\frac{\cos nt_k-1}{2\tg (t_k/2)}f({\bf x}+{\bf t})d{\bf t}
\end{align*}
in the proof of the theorem. 
\begin{lemma}\label{L2}
If $f\in\Log_{d-1}(L)(\ZT^d)$, then
\begin{align}
&\int_{\ZT^d}\sup_n\left|S_{n}f({\bf x})-S^*_{n}f({\bf x})\right|d{\bf x}\lesssim  \|f\|_{\Log_{d-1}(L)(\ZT^d)},\label{a26}\\
&\int_{\ZT^d}\sup_n\left|\tilde S_{n}f({\bf x})-\tilde S^*_{n}f({\bf x})\right|d{\bf x}\lesssim  \|f\|_{\Log_{d-1}(L)(\ZT^d)}.\label{a38}
\end{align}
\end{lemma}
\begin{proof}
Obviously, one can suppose that
\begin{equation}\label{a76}
\|f\|_{\Log_{d-1}(L)(\ZT^d)}=1.
\end{equation}
We shall only prove \e {a26}. The inequality \e {a38} can be proved similarly. We have
\begin{equation}\label{a20}
S_nf({\bf x})=\frac{1}{\pi^d}\int_{\ZT^d}\prod_{k=1}^dD_n(t_k)f({\bf x}+{\bf t})d{\bf t},
\end{equation} 
where 
\begin{equation}\label{a22}
D_n(x)=\frac{\sin(n+1/2)x}{2\sin(x/2)}=\frac{\sin nx}{2\tg (x/2)}+\frac{1}{2}\cos nx,
\end{equation}
is the Dirichlet kernel. Substituting \e {a22} in \e{a20}, the difference
\md0
S_{n}f({\bf x})-S^*_{n}f({\bf x})
\emd
turns to be a sum of several integrals of the form
\begin{equation}\label{a23}
\frac{1}{(2\pi)^d}\int_{\ZT^d}\prod_{k\in A}\frac{\sin nt_k}{\tg(t_k/2)}\prod_{k\in A^c}\cos (nt_k)\cdot f({\bf x}+{\bf t})d{\bf t},
\end{equation}
where $A\subsetneq \{1,2,\ldots, d\}$ is a subset of integers. Then applying the formulas of trigonometric functions products, each integral \e {a23} can be split into the sum of some integrals 
\begin{equation}\label{a24}
\frac{1}{(2\pi)^d}\int_{\ZT^d}\frac{\phi (n(\pm t_1\pm t_2\pm \ldots\pm t_d))}{2^{d-1}\prod_{k\in A}\tg (t_k/2)}\cdot f({\bf x}+{\bf t})d{\bf t},
\end{equation}
where $\phi$ is either sine or cosine function. This reduces the lemma to the estimation of the integrals \e {a24}. The case of $A=\varnothing$ is estimated by
\begin{equation*}
\frac{1}{(2\pi)^d}\int_{\ZT^d}|f({\bf x}+{\bf t})|d{\bf t}=\frac{\|f\|_{L^1}}{(2\pi)^d}\lesssim \|f\|_{L\log^{d-1}L}.
\end{equation*}
The cases $A\neq \varnothing$ of the integrals \e {a24} are estimated similarly. So we can restrict us only on the estimation of 
\begin{equation}\label{a25}
I_nf({\bf x})=\int_{\ZT^d}\frac{\sin n(t_1+ t_2+\ldots+ t_d)}{\prod_{k=l+1}^d\tg (t_k/2)}\cdot f({\bf x}+{\bf t})d{\bf t}
\end{equation}
corresponding to $A=\{1,\ldots,l\}$, $l\ge 1$. After the change of variables 
\begin{equation}\label{a28}
u_1=t_1+t_2+\ldots+t_d,\,  u_2=t_2,\ldots, \, u_d=t_d,
\end{equation}
from \e {a25} we get
\begin{align*}
|I_nf({\bf x})|&=\left|\int_{\ZT^d}\frac{\sin nu_1}{\prod_{k=l+1}^d\tg (u_k/2)}\cdot G({\bf x},{\bf u})d{\bf u}\right|\\
&=\int_{\ZT^{l}}\sin nu_1\left(\int_{\ZT^{d-l}} \frac{G({\bf x},{\bf u})}{\prod_{k=l+1}^d\tg (u_k/2)}\cdot du_{l+1}\ldots du_d\right)du_1\ldots du_{l}\\
&\le \int_{\ZT^{l}}\left|\int_{\ZT^{d-l}} \frac{G({\bf x},{\bf u})}{\prod_{k=l+1}^d\tg (u_k/2)}\cdot du_{l+1}\ldots du_d\right|du_1\ldots du_{l}
\end{align*}
where
\begin{equation}\label{a29}
G({\bf x},{\bf u})=f(x_1+u_1-u_2-\ldots-u_d,x_2+u_2,\ldots,x_d+u_d).
\end{equation}
The inner integral can be considered as a function on variables $x_k$, $k=1,2,\ldots,d$, and $u_j$, $j=1,2,\ldots l$. Moreover, the $(d-l)$-time iteration of \lem {L1} implies
\begin{align*}
\int_{\ZT^{d}}&\sup_n|I_n({\bf x})|d\textbf{x}\\
&\le \int_{\ZT^{d+l}}\left|\int_{\ZT^{d-l}} \frac{G({\bf x},{\bf u})}{\prod_{k=l+1}^d\tg (u_k/2)}\cdot du_{l+1}\ldots du_d\right|du_1\ldots du_l dx_1\ldots dx_d\\
&\lesssim 1+\int_{\ZT^{d+l}}\Log_{d-l}\left(|G({\bf x},u_1,\ldots, u_l,0,\ldots, 0)|\right)du_1\ldots du_l dx_1\ldots dx_d\\
&=1+(2\pi)^l\int_{\ZT^d}\Log_{d-l}(f)\\
&\lesssim \|f\|_{\Log_{d-1}(L)(\ZT^d)}=1
\end{align*}
that gives \e {a26}. Note that in the above estimations we use bound \e {a73} that is valid under assumption \e {a76}. Lemma is proved.
\end{proof}
\section{Proofs of main results}
\begin{proof}[Proof of \trm {T1}]
First we shall prove estimate \e {a40} for the operators
	\begin{equation}\label{a31}
	U_nf(\textbf{x})=\frac{1}{\pi^d}\int_{\ZT^d}\prod_{k=1}^d\frac{\phi_k(t_k)}{2\tg (t_k/2)}f({\bf x}+{\bf t})d{\bf t}.
	\end{equation} 
where $\phi_k$ is either sine or cosine function. We call them $U$-type operators. We do it by induction on dimension.
In the one dimensional case we have either
\begin{equation}
U_nf(x)=\frac{1}{\pi}\int_{\ZT}\frac{\sin nt}{2\tg (t/2)}f( x+t)dt
\end{equation}
or the same with the cosine function. Thus we get
\begin{align*}
U_nf(x)&=\frac{1}{\pi}\int_{\ZT}\frac{\sin n(t-x)}{2\tg ((t-x)/2)}f(t)dt\\
&= \frac{\cos nx}{\pi}\int_{\ZT}\frac{\sin nt\cdot f(t)}{2\tg ((t-x)/2)}dt-\frac{\sin nx}{\pi}\int_{\ZT}\frac{\cos nt\cdot f(t)}{2\tg ((t-x)/2)}dt 
\end{align*} 
The last integrals are conjugates of functions $\sin nt\cdot f(t)$ and $\cos nt\cdot f(t)$ respectively. Besides, their maximal functions can be dominated by $Mf(x)$. Thus, taking $F(x)=Mf(x)$, from inequality \e {a39} we conclude 
\begin{equation*}
\int_\ZT\exp\left(c_1\frac{U_nf(x)}{F(x)}\right)dx<c_2.
\end{equation*}
On the other hand weak-$L^1$ inequality of maximal function implies 
\begin{equation*}
\{x\in\ZT:\,F(x)>\lambda \}\lesssim \frac{\|f\|_{L^1}}{\lambda}
\end{equation*}
which completes the estimate in one dimension.

 Suppose that the exponential estimate for operators \e {a31} in $d-1\ge 2$ dimension holds. Take a function $f\in \Log_{d-1}(\ZT^d)$ satisfying
 \begin{equation}\label{a74}
 	\| f\|_{\Log_{d-1}(L)(\ZT^d)}=1.
 \end{equation} 
 Then we have
\begin{align*}
U_nf({\bf x})&=\frac{1}{\pi^d}\int_{\ZT^d}\prod_{k=1}^d\frac{\phi_k (nt_k)}{2\tg (t_k/2)}\cdot f({\bf x}+{\bf t})d{\bf t}\\
&=\frac{1}{\pi^d}\int_{\ZT^d}\prod_{k=1}^{d-2}\frac{\phi_k (nt_k)}{2\tg (t_k/2)}\cdot \frac{\phi_{d-1} (nt_{d-1})\cdot \phi_d (nt_d)}{4\tg (t_{d-1}/2)\tg (t_d/2)}\cdot f({\bf x}+{\bf t})d{\bf t}.
\end{align*} 
Without loss of generality we can suppose that $\phi_{d-1}$ and $\phi_d$ are both sine functions. Thus we obtain
\begin{align*}
U_nf({\bf x})&=\frac{1}{2\pi^d}\int_{\ZT^d}\prod_{k=1}^{d-2}\frac{\phi_k (nt_k)}{2\tg (t_k/2)}\cdot \frac{\cos n(t_{d-1}-t_d)}{4\tg (t_{d-1}/2)\tg (t_d/2)}\cdot f({\bf x}+{\bf t})d{\bf t}\\
&\qquad -\frac{1}{2\pi^d}\int_{\ZT^d}\prod_{k=1}^{d-2}\frac{\phi_k (nt_k)}{2\tg (t_k/2)}\cdot \frac{\cos n(t_{d-1}+t_d)}{4\tg (t_{d-1}/2)\tg (t_d/2)}\cdot f({\bf x}+{\bf t})d{\bf t}\\
&=U_n^{(1)}f({\bf x})-U_n^{(2)}f({\bf x}).
\end{align*} 
We will estimate only the first integral $U_n^{(1)}f({\bf x})$. The second can be evaluated similarly. 
Performing the change of variables
\begin{equation*}
u_1=t_1,\,u_2=t_2,\ldots, u_{d-1}=t_{d-1}-t_d, u_d=t_d,
\end{equation*}
in the expression of $U_n^{(1)}f({\bf x})$, we obtain
\begin{equation*}
U_n^{(1)}f({\bf x})=\frac{1}{2\pi^d}\int_{\ZT^d}\prod_{k=1}^{d-2}\frac{\phi_k (nu_k)}{2\tg (u_k/2)}\cdot \frac{\cos nu_{d-1}}{4\tg ((u_{d-1}+u_d/2))\tg (u_d/2)}\cdot G({\bf x},{\bf u})d{\bf u},
\end{equation*}
where
\begin{equation}\label{a33}
G({\bf x},{\bf u})=f(x_1+u_1,\,\ldots,\,x_{d-2}+u_{d-2},\,x_{d-1}+u_{d-1}+u_d,\,x_d+u_d).
\end{equation}
Applying the identity
\md0
\frac{1}{\tg (u+v)\tg v }=\frac{1}{\tg u\tg v}-\frac{1}{\tg u\tg (u+v)}-1,
\emd
we obtain
\begin{align*}
U_n^{(1)}f({\bf x})&=\frac{1}{2\pi^d}\int_{\ZT^d}\prod_{k=1}^{d-2}\frac{\phi_k (nu_k)}{2\tg (u_k/2)}\cdot \frac{\cos nu_{d-1}}{2\tg (u_{d-1}/2)}\cdot \frac{1}{2\tg (u_d/2)}\cdot G({\bf x},{\bf u})d{\bf u}\\
&-\frac{1}{2\pi^d}\int_{\ZT^d}\prod_{k=1}^{d-2}\frac{\phi_k (nu_k)}{2\tg (u_k/2)}\cdot \frac{\cos nu_{d-1}}{2\tg (u_{d-1}/2)}\cdot \frac{1}{2\tg ((u_{d-1}+u_d)/2)}\cdot G({\bf x},{\bf u})d{\bf u}\\
&-\frac{1}{2\pi^d}\int_{\ZT^d}\prod_{k=1}^{d-2}\frac{\phi_k (nu_k)}{2\tg (u_k/2)}\cdot \cos nu_{d-1}\cdot G({\bf x},{\bf u})d{\bf u}\\
&=U_n^{(1,1)}f({\bf x})-U_n^{(1,2)}f({\bf x})-U_n^{(1,3)}f({\bf x}).
\end{align*}
For each $i=1,2,3$ we shall find a function $F^{(i)}(\textbf{x})\ge 0$ such that 
\begin{align}
&|\{\textbf{x}\in \ZT^{d}:\, F^{(i)}(\textbf{x})>\lambda \}|\lesssim \frac{\|f\|_{\Log_{d-1}(\ZT^{d})}}{\lambda},\label{a45}\\
&\int_{\ZT^{d}}\exp\left(\frac{\varepsilon |U_n^{(1,i)}f({\bf x})|}{F^{(i)}(\textbf{x})}\right)d\textbf{x}\lesssim 1.\label{a46}
\end{align}
\textbf{Case $i=1$:} Consider the operator
\begin{align*}
U'_ng(x_1,\ldots,x_d)&=\frac{1}{2\pi^{d-1}}\int_{\ZT^{d-1}}\prod_{k=1}^{d-2}\frac{\phi_k (nu_k)}{2\tg (u_k/2)}\cdot \frac{\cos nu_{d-1}}{2\tg (u_{d-1}/2)}\\
&\qquad\times g(x_1+u_1,\ldots,x_{d-1}+u_{d-1},x_d)du_1\ldots du_{d-1}.
\end{align*}
applied on the function
\begin{equation}\label{a34}
g(x_1,\ldots,x_d)=\frac{1}{\pi}\int_{\ZT} \frac{f(x_1,\,\ldots,\,x_{d-2},\,x_{d-1}+t,\,x_d+t)}{2\tg (t/2)} dt.
\end{equation}
Taking into account \e {a33}, we get
\begin{align}
U_n^{(1,1)}f({\bf x})&=\frac{1}{2\pi^{d-1}}\int_{\ZT^{d-1}}\prod_{k=1}^{d-2}\frac{\phi_k (nu_k)}{2\tg (u_k/2)}\cdot \frac{\cos nu_{d-1}}{2\tg (u_{d-1}/2)}\label{a42}\\
&\qquad\times\left(\frac{1}{\pi}\int_{\ZT} \frac{1}{2\tg (u_d/2)}\cdot G({\bf x},{\bf u})du_d\right)du_1,\ldots du_{d-1}\nonumber\\
&=\frac{1}{2\pi^{d-1}}\int_{\ZT^{d-1}}\prod_{k=1}^{d-2}\frac{\phi_k (nu_k)}{2\tg (u_k/2)}\cdot \frac{\cos nu_{d-1}}{2\tg (u_{d-1}/2)}\nonumber\\
&\qquad\times g(x_1+u_1,\ldots,x_{d-1}+u_{d-1},x_d)du_1\ldots du_{d-1}\nonumber\\
&=U'_ng(x_1,\ldots,x_{d-1},x_d).\nonumber
\end{align} 
For a fixed $x_d$ the operator $U'_n$ can be considered as a $(d-1)$-dimensional $U$-type operator \e {a31}. 
Thus, according to the induction hypothesis, for each $x_d\in\ZT$ one can find a function $F_{x_d}(x_1,\ldots,x_{d-1})=F^{(1)}(x_1,\ldots,x_d)$ such that
\begin{align}
&|\{(x_1,\ldots,x_{d-1})\in \ZT^{d-1}:\, F_{x_d}(x_1,\ldots,x_{d-1})>\lambda \}|\lesssim \frac{\|g_{x_d}\|_{\Log_{d-2}(\ZT^{d-1})}}{\lambda},\label{a41}\\
&\int_{\ZT^{d-1}}\exp\left(\frac{c\varepsilon |U'_ng_{x_g}(x_1,\ldots,x_{d-1})|}{F_{x_d}(x_1,\ldots,x_{d-1})}\right)dx_1\ldots d x_{d-1}\lesssim 1,\, n=1,2,\ldots.\label{a44}
\end{align}
Here $g_{x_d}$ is $g(x_1,\ldots,x_d)$ as a function of variables $x_1,\ldots, x_{d-1}$.
On the other hand from \lem {L1} it follows that 
\begin{equation}\label{a43}
\int_{\ZT^{d}}\Log_{d-2}(g)\lesssim 1+\int_{\ZT^{d}}\Log_{d-1}(f)\lesssim \|f\|_{\Log_{d-1}(\ZT^{d})}=1.
\end{equation}
Applying \e {a72}, \e {a73}, \e {a43} and \e {a41},  we obtain
\begin{align*}
|\{\textbf{x}\in \ZT^{d}:\, F^{(1)}(\textbf{x})>\lambda \}|&\lesssim \frac{1}{\lambda}\int_{\ZT}\|g_{x_d}\|_{\Log_{d-2}(\ZT^{d-1})}dx_d\\
&\lesssim\frac{1}{\lambda}\int_{\ZT}\left(1+\int_{\ZT^{d-1}}\Log_{d-2}(g)dx_1\ldots dx_{d-1}\right)dx_d\\
&\lesssim \frac{1}{\lambda}\left(1+\int_{\ZT^{d}}\Log_{d-2}(g)\right)\\
&\lesssim \frac{1}{\lambda}\left(1+\int_{\ZT^{d}}\Log_{d-1}(f)\right)\\
&\lesssim \frac{\|f\|_{\Log_{d-1}(\ZT^{d})}}{\lambda}.
\end{align*}
Using \e{a42} and integrating inequality \e {a44} with respect to the variable $x_d$, we get
\begin{equation*}
\int_{\ZT^{d}}\exp\left(\frac{c\varepsilon |U_n^{(1,1)}f({\bf x})|}{F^{(1)}(\textbf{x})}\right)d\textbf{x}\lesssim 1
\end{equation*}
Thus we get \e {a45} and \e {a46} for $i=1$.

\textbf{Case $i=2$}: The estimation of $U_n^{(1,2)}f({\bf x})$ is based on the same argument. We have
\begin{align*}
U_n^{(1,2)}f({\bf x})&=\frac{1}{2\pi^{d-1}}\int_{\ZT^{d-1}}\prod_{k=1}^{d-2}\frac{\phi_k (nu_k)}{2\tg (u_k/2)}\cdot \frac{\cos nu_{d-1}}{2\tg (u_{d-1}/2)}du_1\ldots du_d\\
&\qquad \times\frac{1}{\pi}\int_{\ZT}\frac{G({\bf x},{\bf u})}{2\tg ((u_{d-1}+u_d)/2)}du_d.
\end{align*}
Change of variable $t=u_d+u_{d-1}$ the inner integral implies
\begin{align*}
\frac{1}{\pi}\int_{\ZT}&\frac{G({\bf x},{\bf u})}{2\tg ((u_{d-1}+u_d)/2)}du_d\\
&=\frac{1}{\pi}\int_{\ZT}\frac{f(x_1+u_1,\ldots,x_{d-2}+u_{d-2}, x_{d-1}+t,x_d-u_{d-1}+t)}{2\tg (t/2)}dt\\
&=g(x_1+u_1,\ldots,x_{d-2}+u_{d-2},x_{d-1},x_d-u_{d-1}),
\end{align*}
where $g$ is the same function \e {a34}. Thus we obtain
\begin{align*}
U_n^{(1,2)}f({\bf x})&=U''_ng(x_1,\ldots,x_{d-1},x_d)\\
&=\frac{1}{2\pi^{d-1}}\int_{\ZT^{d-1}}\prod_{k=1}^{d-2}\frac{\phi_k (nu_k)}{2\tg (u_k/2)}\cdot \frac{\cos nu_{d-1}}{2\tg (u_{d-1}/2)}\\
&\qquad\times g(x_1+u_1,\ldots,x_{d-2}+u_{d-2},x_{d-1}, x_d-u_{d-1})du_1\ldots du_{d-1}.
\end{align*}
For a fixed $x_{d-1}$ it can be considered as a $(d-1)$-dimensional $U$-type operator applied to $g$ as a function of the rest variables $x_1,\ldots, x_{d-2},x_d$. Applying the induction hypothesis, likewise the case of $i=1$, we then get a function $F^{(2)}(\textbf{x})$ satisfying \e {a45} and \e {a46} for $i=2$.

\textbf{Case $i=3$}: Write $U_n^{(1,3)}$ in the form
\begin{align*}
U_n^{(1,3)}f({\bf x})&=\frac{1}{2\pi^{d-1}}\int_{\ZT^{d-1}}\prod_{k=1}^{d-2}\frac{\phi_k (nu_k)}{2\tg (u_k/2)}\cdot \frac{\cos nu_{d-1}}{2\tg (u_k/2)}\\
&\qquad\times \left(\frac{1}{\pi}\int_{\ZT}\frac{G({\bf x},{\bf u})}{2\ctg (u_d/2)}du_d\right)d u_1\ldots du_{d-1}.
\end{align*}
After the change of variable $t=u_d-\pi$ for the inner integral we get the form
\begin{align*}
\frac{1}{\pi}\int_{\ZT}&\frac{f(x_1+u_1,\,\ldots,\,x_{d-2}+u_{d-2},\,x_{d-1}+u_{d-1}+u_d,\,x_d+u_d)}{2\ctg (u_d/2)}du_d\\
&=\frac{1}{\pi}\int_{\ZT}\frac{f(x_1+u_1,\,\ldots,\,x_{d-2}+u_{d-2},\,x_{d-1}+u_{d-1}+\pi+t,\,x_d+\pi+t)}{2\tg (t/2)}dt\\
&=g(x_1+u_1,\ldots,x_{d-1}+u_{d-2},x_{d-1}+u_{d-1}+\pi,x_d+\pi),
\end{align*} 
where $g$ is the function \e {a34}. Thus we obtain that
\begin{align*}
U_n^{(1,3)}f({\bf x})&=U'''_ng(x_1,\ldots,x_{d-1},x_d)\\
&=\frac{1}{2\pi^{d-1}}\int_{\ZT^{d-1}}\prod_{k=1}^{d-2}\frac{\phi_k (nu_k)}{2\tg (u_k/2)}\cdot \frac{\cos nu_{d-1}}{2\tg (u_{d-1}/2)}\\
&\qquad\times g(x_1+u_1,\ldots,x_{d-2}+u_{d-2},x_{d-1}+u_{d-1}+\pi, x_d+\pi)du_1\ldots du_{d-1}
\end{align*}
is a $(d-1)$-dimensional $U$-operator applied to the function
\begin{equation}\label{a35}
g(x_1,\ldots,x_{d-2},x_{d-1}+\pi, x_d+\pi)
\end{equation}
as a function on variables $x_1,\ldots,x_{d-1}$. Similarly we get \e {a45} and \e {a46} for $i=3$.	

Hence the desired estimation of $U_n$ is complete.

Since $S_n^*$ is a $U$-operator we can find a function $F_1(\textbf{x})$ such that
\begin{align}
&|\{\textbf{x}\in \ZT^{d}:\, F_1(\textbf{x})>\lambda \}|\lesssim \frac{\|f\|_{\Log_{d-1}(\ZT^{d})}}{\lambda}\label{a48}\\
&\int_{\ZT^{d}}\exp\left(\frac{\varepsilon |S^*_nf({\bf x})|}{F_1(\textbf{x})}\right)d\textbf{x}\lesssim 1,\quad n=1,2,\ldots .\label{a49}
\end{align}
As for the $\tilde S^*_n$, we have
\begin{equation*}
|\tilde S_n^*f({\bf x})|= |U_nf(\textbf{x})|+
G(\textbf{x}),
\end{equation*}
where 
\begin{align*}
&U_nf(\textbf{x})=\frac{1}{\pi^d}\int_{\ZT^d}\prod_{k=1}^d\frac{\cos nt_k}{2\tg (t_k/2)}f({\bf x}+{\bf t})d\textbf{t},\\
&G(\textbf{x})=\frac{1}{\pi^d}\left|\int_{\ZT^d}\frac{f({\bf x}+{\bf t})}{\prod_{k=1}^d2\tg (t_k/2)}d{\bf t}\right|,
\end{align*}
and $G(\textbf{x})$ satisfies
\begin{equation}\label{a50}
|\{G(\textbf{x})>\lambda \}|\lesssim \frac{\|f\|_{\Log_{d-1}(\ZT^d)}}{\lambda}.
\end{equation}
Since $U_n$ is a $U$-operator, there is a function $F_2(\textbf{x})$ satisfying
\begin{align}
&|\{\textbf{x}\in \ZT^{d}:\, F_2(\textbf{x})>\lambda \}|\lesssim \frac{\|f\|_{\Log_{d-1}(\ZT^{d})}}{\lambda},\label{a51}\\
&\int_{\ZT^{d}}\exp\left(\frac{\varepsilon |U_nf({\bf x})|}{F_2(\textbf{x})}\right)d\textbf{x}\lesssim 1,\quad n=1,2,\ldots .\label{a52}
\end{align}
Finally, according to \lem {L2} we have
\begin{align}
|S_nf({\bf x})|+|\tilde S_nf({\bf x})|&\le |S^*_nf({\bf x})|+|\tilde S_n^*f({\bf x})|+ F_3({\bf x})\\
&\le |S^*_nf({\bf x})|+|U_nf({\bf x})|+ G(\textbf{x})+F_3({\bf x}),
\end{align}
where the function $F_3({\bf x})\ge 0$ satisfies
\begin{equation}\label{a53}
\|F_3\|_{L^1(\ZT^d)}\lesssim \|f\|_{\Log_{d-1}(L)(\ZT^d)}.
\end{equation}
Observe that the function $F=4(F_1+F_2+F_3+G)$ will satisfy the conditions of \trm {T1}. Indeed, \e {a47} immediately follows from \e {a48}, \e {a50}, \e {a51} and \e {a53} (for $F_3$ we additionally apply Chebyshev's inequality). To prove \e {a40} observe that
\begin{align*}
\exp&\left(\frac{|S_nf(\textbf{x})|+|\tilde S_nf(\textbf{x})|}{F(\textbf{x})}\right)\\
&\qquad\le \exp\left(\frac{|S^*_nf({\bf x})|+|U_nf({\bf x})|+G(\textbf{x})+F_3({\bf x})}{F(\textbf{x})}\right)\\
&\qquad \le \exp\left(\frac{4|S^*_nf({\bf x})|}{F(\textbf{x})}\right) +\exp\left(\frac{4|U_nf({\bf x})|}{F(\textbf{x})}\right)\\
&\qquad\qquad+\exp\left(4\frac{G(\textbf{x})}{F(\textbf{x})}\right) 
+\exp\left(4\frac{F_3({\bf x})}{F(\textbf{x})}\right)\\
&\qquad \le\exp\left(\frac{|S^*_nf({\bf x})|}{F_1(\textbf{x})}\right) +\exp\left(\frac{|U_nf({\bf x})|}{F_2(\textbf{x})}\right)+2e.
\end{align*} 
Combining this with \e {a49} and \e {a52}, we will complete the proof of theorem.
\end{proof}
\begin{proof}[Proof of \coro {C1}]
Let $f\in \Log_{d-1}(L)(\ZT^d)$ and $F(x)$ be the function satisfying the conditions of \trm {T1}. Define
\begin{equation*}
	E_{f,\varepsilon}=\left\{\textbf{x}\in\ZT^d:\, F(\textbf{x})\le\frac{\|f\|_{\Log_{d-1}(\ZT^d)}}{\gamma\varepsilon}\right\},
\end{equation*}
where $\gamma$ is a constant. According to \e {a47} there is a constant $\gamma$ depended only on $d$ such that $|(E_{f,\varepsilon})^c|<\varepsilon $. This implies \e {a54}. Besides from \e {a40} we obtain
\begin{multline*}
\int_{E_{f,\varepsilon}}\exp\left(\gamma\cdot\varepsilon\cdot\frac{ |S_nf({\bf x})|+|\tilde S_nf({\bf x})|}{\|f\|_{\Log_{d-1}(\ZT^d)}}\right)d{\bf x}\\
\le \int_{\ZT^d}\exp\left(\frac{ |S_nf({\bf x})|+|\tilde S_nf({\bf x})|}{F(\textbf{x})}\right)d{\bf x}\lesssim 1,
\end{multline*}
completing the proof of corollary.
\end{proof}
\begin{proof}[Proof of \coro {C2}]
Given $f\in \Log_{d-1}(L)(\ZT^d)$. It is well known that $(C,1)$ means $\sigma_{\bf n}f$ of the Fourier series \e {x1} of $f$  and its conjugate \e {x2} almost everywhere converge to $f$ and $\tilde f$ respectively. Besides, there is norm convergence 
\begin{equation*}
\lim_{\min(\textbf{n})\to\infty}\|\sigma_{\bf n}f-f\|_{\Log_{d-1}(\ZT^d)}=0.
\end{equation*}
Applying this, one can find a set $G\subset \ZT^d$ and $d$-dimensional trigonometric polynomial $P_k$ such that
\begin{align}
&|G|>(2\pi)^d-\varepsilon/2,\label{a60}\\
 &\|f-P_k\|_{L^\infty(G)}<1/2k,\label{a61}\\
 &\|\tilde f-\tilde P_k\|_{L^\infty(G)}<1/2k,\label{a62}\\
 &\|f-P_k\|_{\Log_{d-1}(\ZT^d)}<\gamma\varepsilon_k/2k.\label{a63}
\end{align}
Applying \coro {C1} for $\varepsilon_k=\varepsilon/2^{k+1}$, we find sets $E_k\subset \ZT^d$ so that
\begin{align}
&|E_k|>(2\pi)^d-\varepsilon_k,\quad k=1,2,\ldots,\label{a55}\\
&\int_{E_k}\exp\left(\gamma\varepsilon_k\frac{ |S_n(f-P_k)|+|\tilde S_n(f-P_k)|}{\|f-P_k\|_{\Log_{d-1}(\ZT^d)}}\right)\le c,\, n=1,2,\ldots.\label{a56}
\end{align}
Define
\begin{equation*}
E_{f,\varepsilon}=G\bigcap\left(\bigcap_{k}E_k\right).
\end{equation*}
From \e {a60} and \e {a55} it follows the condition \e {a54}. Let $\phi(t)=\exp t-1$. One can easily check that $\phi(ab)\le a\phi (b)$ for $0<a <1$ and $b>0$. Thus, applying \e {a61}, \e {a63} and \e {a56}, we get
\begin{align*}
\lim_{n\to\infty}&\int_{E_{f,\varepsilon}}\bigg(\exp\big(A|S_nf-f|\big)-1\bigg)\\
&=\lim_{n\to\infty}\int_{E_{f,\varepsilon}}\bigg(\exp\big(A|S_n(f-P_k)-(f-P_k)|\big)-1\bigg)	\\
&\le\frac{A}{k}\sup_{n}\int_{E_{f,\varepsilon}}\bigg(\exp\big(k(|S_n(f-P_k)|+|f-P_k|)\big)\bigg)\\
&\le \frac{A}{k}\left(\sup_{n} \int_{E_{f,\varepsilon}}\exp\left(2k|S_n(f-P_k)|\right)+\int_{E_{f,\varepsilon}}\exp\left(2k|f-P_k|\right)\right)\\
&\le \frac{A}{k}\left(\sup_{n} \int_{E_{f,\varepsilon}}\exp\left(\frac{\gamma\varepsilon_k|S_n(f-P_k)|}{\|f-P_k\|_{\Log_{d-1}(\ZT^d)}}\right)+\int_{E_{f,\varepsilon}}\exp\left(2k|f-P_k|\right)\right)\\
&\lesssim  \frac{A}{k}.
\end{align*}
Since the last quantity can be arbitrarily small, we get \e {a58}. Similarly, we can get
\begin{align*}
\lim_{n\to\infty}&\int_{E_{f,\varepsilon}}\bigg(\exp\left(A|\tilde S_nf-\tilde f|\right)-1\bigg)\\
&\le \frac{A}{k}\left(\sup_{n} \int_{E_{f,\varepsilon}}\exp\left(\frac{\gamma\varepsilon_k|\tilde S_n(f-P_k)|}{\|f-P_k\|_{\Log_{d-1}(\ZT^d)}}\right)+\int_{E_{f,\varepsilon}}\exp\left(2k|\tilde f-\tilde P_k|\right)\right)\\
&\lesssim  \frac{A}{k},
\end{align*}
and so \e {a59}. 
\end{proof}
\begin{proof}[Proof of \coro {C3}]
Let $n_k$, $k=1,2,\ldots$, be a sequence of positive integers and $f\in \Log_{d-1}(L)(\ZT^d)$. Given $\varepsilon>0$, we can find a polynomial $P$ such that
\begin{equation}
\|f-P\|_{\Log_{d-1}(\ZT^d)}<\gamma\varepsilon/2.
\end{equation}
Applying \coro {C1}, we find a set $E_{f,\varepsilon}\subset \ZT^d$ such that
\begin{align}
&|E_{f,\varepsilon}|>(2\pi)^d-\varepsilon,\label{a64}\\
&\int_{E_{f,\varepsilon}}\exp\left(\gamma\varepsilon\frac{ |S_n(f-P)|+|\tilde S_n(f-P)|}{\|f-P\|_{\Log_{d-1}(\ZT^d)}}\right)\le c,\, n=1,2,\ldots.\label{a65}
\end{align}	
Denote
\begin{align*}
E_k&=\{\textbf{x}\in E_{f,\varepsilon}:\,|S_{n_k}(f-P)(\textbf{x})|+|\tilde S_{n_k}(f-P)(\textbf{x})|>\varepsilon\log k \}\\
&=\left\{\textbf{x}\in E_{f,\varepsilon}:\,\exp\left(\gamma\varepsilon\frac{ |S_{n_k}(f-P)|+|\tilde S_{n_k}(f-P)|}{\|f-P\|_{\Log_{d-1}(\ZT^d)}}\right)\right.\\
&\qquad\qquad\qquad\qquad\qquad\qquad\left.>\exp\left(\frac{\gamma \varepsilon}{\|f-P\|_{\Log_{d-1}(\ZT^d)}}\log k\right) \right\}\\
&\subset \left\{\textbf{x}\in E_{f,\varepsilon}:\,\exp\left(\gamma\varepsilon\frac{ |S_{n_k}(f-P)|+|\tilde S_{n_k}(f-P)|}{\|f-P\|_{\Log_{d-1}(\ZT^d)}}\right)>k^2\right\}.
\end{align*}
Thus, applying Chebishev's inequality, from \e {a65} we get $|E_k|\le c|E_{f,\varepsilon}|/k^2$, and so for almost all 
$\textbf{x}\in E_{f,\varepsilon}$ we have $\textbf{x}\in E_k$, $k>k(\textbf{x})$. This implies that
\begin{equation*}
\limsup_{k\to\infty}\frac{ |S_{n_k}(f)|+|\tilde S_{n_k}(f)|}{\log k}\le \varepsilon 
\end{equation*}
a.e. on $E_{f,\varepsilon}$. Since $\varepsilon>0$ can be arbitrarily small we get \e {a66}.
\end{proof}

\end{document}